\documentclass[11pt]{amsart}
\usepackage{mathrsfs}
\usepackage{amsfonts}
\usepackage{amsthm}
\usepackage{amssymb}
\usepackage{latexsym}
\newtheorem{theorem}{Theorem}
\newtheorem{lemma}{Lemma}
\newtheorem{corollary}{Corollary}

\newtheorem{proposition}{Proposition}

\begin{document}

\title[Sobolev estimates for the complex Green operator]{Sobolev estimates for the complex
Green operator on CR submanifolds of hypersurface type}

\author{Emil J.~Straube}
\author{Yunus E.~Zeytuncu}

\address{Department of Mathematics Texas A\&M University College Station, Texas, 77843}
\address{Department of Mathematics and Statistics, University of Michigan-Dearborn, Dearborn, Michigan, 48128}
\email{straube@math.tamu.edu}
\email{zeytuncu@umich.edu}

\thanks{Research supported in part by NSF grant DMS 0758534.}

\date{November 13, 2013; revised October 30, 2014}

\begin{abstract}
Let $M$ be a pseudoconvex, oriented, bounded and closed CR submanifold of $\mathbb{C}^{n}$ of hypersurface type. Our main result says that when a certain $1$--form on $M$ is exact on the null space of the Levi form, then the complex Green operator on $M$ satisfies Sobolev estimates. This happens in particular when $M$ admits a set of plurisubharmonic defining functions or when $M$ is strictly pseudoconvex except for the points on a simply connected complex submanifold.
\end{abstract}

\maketitle

\section{Introduction}\label{intro}

A (connected) CR submanifold $M$ of $\mathbb{C}^{n}$ is called of hypersurface type if at
each point of $M$, the complex tangent space has co--dimension one inside the real tangent
space.
The $\overline{\partial}$--complex in $\mathbb{C}^{n}$
induces the (extrinsic) $\overline{\partial}_{M}$--complex on $M$ (\cite{Boggess91,
ChenShaw01}). Throughout this paper, we assume that $M$ is orientable and compact
(without boundary). We will study
Sobolev estimates for the $\overline{\partial}_{M}$--complex and the associated complex
Green operator on (oriented) pseudoconvex CR submanifolds of hypersurface type.

Subelliptic estimates for the complex Green operator on (abstract) CR manifolds go
back to \cite{Kohn65} when $M$ is strictly pseudoconvex and to \cite{Kohn85} when $M$ is
of finite type. Compactness estimates for $\overline{\partial}_{M}$ on compact CR
submanifolds of hypersurface type are more recent and may be found in
\cite{RaichStraube08, Raich10, Straube10, MunasingheStraube12}. But Sobolev estimates
hold in many situations where compactness fails. When $M$ is the boundary of a smooth
bounded pseudoconvex
domain in $\mathbb{C}^{n}$, Boas and the first author proved Sobolev estimates for the
complex Green operator on a
large class of (weakly) pseudoconvex boundaries, namely those that admit a
defining function that is plurisubharmonic at points of $M$ (\cite{BoasStraube91a,
BoasStraube91}). In this paper, we prove an analogue of this
result for CR submanifolds of hypersurface type of codimension greater than one. In addition to nontrivial technical issues that must be addressed, an interesting twist
arises that is absent in the case of codimension one. This new difficulty is resolved via the use of results from CR geometry.

When proving Sobolev estimates, one has to control commutators of vector fields with
$\overline{\partial}_{M}$ and $\overline{\partial}_{M}^{*}$; the key is to find vector
fields so that the relevant components of the commutators are small. These components are
conveniently expressed by a 1-form $\alpha$ (\cite{D'Angelo93}, subsection 3.1.1,
\cite{BoasStraube93}, \cite{Straube10a}, section 5.9), and one is led to the condition
that $\alpha$ be
`exact on the null space of the Levi form' as a sufficient condition for Sobolev estimates
for the complex Green operator.
When $M$ is given by plurisubharmonic defining functions, this condition turns out to be satisfied.
It is in the proof of this fact that the ideas from
\cite{BoasStraube91a,
BoasStraube91} do not suffice and CR geometry enters into the argument.

Focusing on
$\alpha$ rather than on the vector fields directly has the additional advantage that we
obtain, essentially with no additional work, the analogue of a result for the
$\overline{\partial}$--Neumann operator from \cite{BoasStraube93}. A special
case of this analogue says that when $M$ is strictly pseudoconvex, except for a complex
submanifold with trivial first DeRham cohomology (smooth as a manifold with boundary),
then the complex Green operators on $M$ are continuous in Sobolev
spaces. More generally, $\alpha$ defines a DeRham cohomology class on complex
submanifolds of $M$, and estimates hold as soon as this class vanishes.

The remainder of the paper is organized as follows. In section \ref{results}, we recall the $\mathcal{L}^{2}$ theory of $\overline{\partial}_{M}$, we introduce the 1-form $\alpha$ mentioned above,
and we state our results. In section \ref{plush-proof}, we prove Theorem \ref{plush} to the effect that
$\alpha$ is exact on the null space of the Levi form when $M$ admits a set of plurisubharmonic defining functions. Section \ref{submanifold-proof} establishes the same conclusion when $M$ is strictly pseudoconvex except for the points on a special submanifold (this is Theorem \ref{submanifold}). The main result, Sobolev estimates for the complex Green operator when $\alpha$ is exact on the null space of the Levi form (Theorem \ref{main}), is proved in two steps. First, we show in section \ref{alpha-fields}
that exactness of $\alpha$ on the null space of the Levi form implies the existence of a family of vector fields with good commutation properties with $\overline{\partial}_{M}$ and $\overline{\partial}_{M}^{*}$
(Proposition \ref{commutators}). Then, in section \ref{fields&reg}, we prove that the existence of such a family of vector fields implies the Sobolev estimates in Theorem \ref{main} (Theorem \ref{fields-estimates}).

\section{Statement of results}\label{results}

We first recall the (standard) setup and notation by giving a
condensed and updated version of section 2 in \cite{Straube10} and section 1 in
\cite{MunasingheStraube12}. The
CR--dimension of $M$, denoted by $m-1$, is the dimension over $\mathbb{C}$ of
$T^{1,0}(M)$. Because $M$ is orientable, there is a global purely imaginary vector field
$T$ on $M$ of unit length that is orthogonal to the complex tangent space
$T^{\mathbb{C}}_{z}(M)$ at every point $z$ in $M$. The Levi form at $z \in M$ is the
Hermitian form $\lambda_{z}$ given by
\begin{equation}\label{levi}
[X,\overline{Y}] = \lambda_{z}(X,\overline{Y})T \;\;\text{mod} \;T^{1,0}(M) \oplus T^{0,1}(M)\,;
\;\;\;X, Y \in T^{1,0}(M).
\end{equation}
$M$ is pseudoconvex if the Levi form is either positive semidefinite at all points, or
negative semidefinite at all points. Replacing $T$ by its negative if necessary, we may
assume that the Levi form is positive semidefinite.

The $\overline{\partial}$--complex in $\mathbb{C}^{n}$ induces the
$\overline{\partial}_{M}$--complex on $M$ (see \cite{Boggess91}, chapter 8,
\cite{ChenShaw01}, chapter 7 for details). If
$L_{1}, \cdots, L_{m-1}$ are local orthonormal sections of
$T^{1,0}(M)$ (in the inner product induced from $\mathbb{C}^{n}$), and $\omega_{1},
\cdots, \omega_{m-1}$ generate the dual basis in $\Lambda^{1,0}(M)$, then a form $u$ can be
written (locally) as
$u = \sum^{\prime}_{|J|=q}u_{J}\overline{\omega}^{J}$, where $\overline{\omega}^{J} =
\overline{\omega_{j_{1}}}\wedge\overline{\omega_{j_{2}}} \cdots
\wedge\overline{\omega_{j_{q}}}$, and the prime indicates summation over strictly
increasing $q$--tuples only. (It will be convenient to still take the coefficients $u_{J}$
to be defined for all $J$ by skew symmetry.) In such a local
frame, $\overline{\partial}_{M}$ is expressed as
\begin{equation}\label{d-bar}
\overline{\partial}_{M}u =
\sideset{}{'}\sum_{|J|=q}\sum_{j=1}^{m-1}(\overline{L_{j}}u_{J})\overline{\omega_{j}}
\wedge\overline{\omega}^{J} +
\sideset{}{'}\sum_{|J|=q}u_{J}\;\overline{\partial}_{M}\overline{\omega}^{J} \;.
\end{equation}
Note that the coefficients of $u$ are not differentiated in the second sum.

The inner product on $(0,q)$--forms on $\mathbb{C}^{n}$ induces a pointwise inner product
on $(0,q)$--forms on $M$. This pointwise inner product provides an $\mathcal{L}^{2}$
inner product on $M$ via integration against (the induced) Lebesgue measure:
\begin{equation}\label{L-2inner}
(u,v) = \int_{M}(u,v)_{z}d\mu_{M}(z) \;.
\end{equation}
We denote by $\mathcal{L}^{2}_{(0,q)}(M)$, $0\leq q\leq (m-1)$, the completion of the
$C^{\infty}$ smooth forms with respect to the norm corresponding to \eqref{L-2inner}.
When expressed in local frames as above, these are precisely the forms all of whose
coefficients are (locally) square integrable. $\overline{\partial}_{M}$ extends to an
(unbounded) operator $\mathcal{L}^{2}_{(0,q)}(M) \rightarrow \mathcal{L}^{2}_{(0,q+1)}(M)$
with domain consisting of those forms where the result, when computed in the sense of
distributions in local frames, is actually in $\mathcal{L}^{2}$. As a densely defined
closed operator, $\overline{\partial}_{M}$ has a Hilbert
space adjoint $\overline{\partial}_{M}^{*}: \mathcal{L}^{2}_{(0,q+1)}(M) \rightarrow
\mathcal{L}^{2}_{(0,q)}(M)$. Integration by parts shows that in local coordinates,
$\overline{\partial}_{M}^{*}$ is given as follows (see for example \cite{ChenShaw01},
section 8.3, \cite{FollandKohn72}, p.94):
\begin{equation}\label{d-bar-*}
\overline{\partial}_{M}^{*}u = -
\sum_{j=1}^{m-1}\sideset{}{'}\sum_{|K|=q-1}L_{j}u_{jK}\overline{\omega}^{K} + \text{terms of
order zero} \;.
\end{equation}
Here, $u_{jK} = u_{j,k_{1}, \cdots, k_{q}}$, and `terms of order zero' means terms that
do not involve derivatives of the coefficients of $u$.

The operator $\overline{\partial}_{M}$ has closed range (at all levels, and hence so does
$\overline{\partial}_{M}^{*}$ at all levels); see \cite{Nicoara06} for $m\geq 3$ (also
\cite{HarringtonRaich10} for a version that weakens the condition on the Levi form to one
that depends on $q$) and \cite{Baracco1, Baracco2} for a proof that also covers $m=2$.
Denote by $\mathcal{H}_{q}(M)$ the subspace of $\mathcal{L}^{2}_{(0,q)}(M)$ consisting of
harmonic forms, that is, of forms in $\ker(\overline{\partial}_{M}) \cap
\ker(\overline{\partial}_{M}^{*})$, and denote by $H_{q}$ the orthogonal projection onto
it. The closed range property implies the estimate
\begin{multline}\label{basic}
\|u\|_{\mathcal{L}^{2}_{(0,q)}(M)}^{2} \lesssim
\|\overline{\partial}_{M}u\|_{\mathcal{L}^{2}_{(0,q+1)}(M)}^{2} +
\|\overline{\partial}_{M}^{*}u\|_{\mathcal{L}^{2}_{(0,q-1)}(M)}^{2} +
\|H_{q}u\|_{\mathcal{L}^{2}_{(0,q)}(M)}^{2} \;,\\
u \in \text{dom}(\overline{\partial}_{M}) \cap \text{dom}(\overline{\partial}_{M}^{*})\;,
\; 0 \leq q \leq m-1\;
\end{multline}
(\cite{Hormander65}, Theorem 1.1.2; \eqref{basic} is actually equivalent to
$\overline{\partial}_{M,q}$ and $(\overline{\partial}_{M,q-1})^{*}$ having closed range).

The Sobolev spaces of $(0,q)$-forms on $M$ are defined in the usual way. Fix a covering of $M$ by coordinate charts so that in each chart, $\omega_{1}, \cdots, \omega_{m-1}$ generate a pointwise orthonormal basis for $(0,1)$--forms. The Sobolev $s$-norm of a form is
computed componentwise in these local frames, via a partition of unity subordinate to the
cover given by the coordinate charts (see for example \cite{FollandKohn72}, p.122).
Derivatives of a form will similarly be taken componentwise in these frames. We use the notation $W^{s}_{(0,q)}(M)$ for the Sobolev space of order $s$ of $(0,q)$-forms.

When $1\leq q\leq (m-2)$, $\mathcal{H}_{q}(M)$ is finite dimensional (\cite{Nicoara06,
HarringtonRaich10}). More precisely, there is the estimate $\|v\|
\lesssim \|v\|_{-1}$ on $\mathcal{H}_{q}(M)$ (\cite{HarringtonRaich10}, proof of Lemma 5.1).
Applying this estimate to $H_{q}u$ gives
\begin{multline}\label{basic1a}
\|H_{q}u\|_{\mathcal{L}^{2}_{(0,q)}(M)} \lesssim \|H_{q}u\|_{W^{-1}_{(0,q)}(M)}\\
\leq \|u\|_{W^{-1}_{(0,q)}(M)}+\|H_{q}u-u\|_{W^{-1}_{(0,q)}(M)}
\lesssim \|u\|_{W^{-1}_{(0,q)}(M)} + \|H_{q}u-u\|_{\mathcal{L}^{2}_{(0,q)}(M)} \\
\lesssim \|u\|_{W^{-1}_{(0,q)}(M)} +\|\overline{\partial}_{M}(H_{q}u-u)\|_{\mathcal{L}^{2}_{(0,q+1)}(M)}  + \|\overline{\partial}_{M}^{*}(H_{q}u-u)\|_{\mathcal{L}^{2}_{(0,q-1)}(M)}  \\
= \|u\|_{W^{-1}_{(0,q)}(M)} +\|\overline{\partial}_{M}u\|_{\mathcal{L}^{2}_{(0,q+1)}(M)} + \|\overline{\partial}_{M}^{*}u\|_{\mathcal{L}^{2}_{(0,q-1)}(M)}  \; ;
\end{multline}
where the estimate in the third line follows from \eqref{basic} (note that $H_{q}(H_{q}u-u)=0$). Reinserting \eqref{basic1a} into \eqref{basic}
gives
\begin{multline}\label{basic1}
\|u\|_{\mathcal{L}^{2}_{(0,q)}(M)}^{2} \lesssim
\|\overline{\partial}_{M}u\|_{\mathcal{L}^{2}_{(0,q+1)}(M)}^{2} +
\|\overline{\partial}_{M}^{*}u\|_{\mathcal{L}^{2}_{(0,q-1)}(M)}^{2} +
\|u\|_{W^{-1}_{(0,q)}(M)}^{2} \;,\\
u \in \text{dom}(\overline{\partial}_{M}) \cap \text{dom}(\overline{\partial}_{M}^{*})\;,
\; 1 \leq q \leq m-2\; .
\end{multline}
The fact that the ranges of $\overline{\partial}_{M}$ and  $\overline{\partial}_{M}^{*}$
are closed (at all levels)
now imply that the complex Laplacian
$\Box = \overline{\partial}_{M}\overline{\partial}_{M}^{*} +
\overline{\partial}_{M}^{*}\overline{\partial}_{M}$, with domain so that the compositions
are defined, maps $(\mathcal{H}_{q})^{\perp} \cap \text{dom}(\Box)$ onto
$(\mathcal{H}_{q})^{\perp}$, and has a bounded inverse $G_{q}$ (on
$(\mathcal{H}_{q})^{\perp}$). Indeed, $G_{q}$ is given by  $\iota_{q}\circ\iota_{q}^{*}$,
where $\iota_{q}$ is the embedding $\iota_{q}: (\mathcal{H}_{q})^{\perp} \cap
\text{dom}(\overline{\partial}_{M}) \cap \text{dom}(\overline{\partial}_{M}^{*})
\rightarrow (\mathcal{H}_{q})^{\perp}$. \eqref{basic} says that
$\iota_{q}$ is continuous. Hence so is its adjoint. Thus $\iota_{q}\circ\iota_{q}^{*}$
is continuous. The argument that it inverts $\Box_{q}$ is the same as in the case of the
$\overline{\partial}$--Neumann operator, see \cite{Straube10a}, proof of
Theorem 2.9, part (1), in particular (2.76)--(2.78) there. It is customary to extend
$G_{q}$ to all of $\mathcal{L}^{2}_{(0,q)}(M)$ by setting it equal to zero on the
kernel of $\Box_{q}$ (which equals $\ker(\overline{\partial}_{M}) \cap
\ker(\overline{\partial}_{M}^{*}$)). $G_{q}$ is `the' complex Green operator.

In addition to the $H_{q}$, we consider two more Szeg\"{o} type projections (note
that $H_{0}$ is the usual Szeg\"{o} projection onto the square integrable CR functions).
Denote by $S_{q}^{\prime}$ the orthogonal projection
$\mathcal{L}^{2}_{(0,q)}(M) \rightarrow \;Im(\overline{\partial}_{M,q-1})$, where
$\overline{\partial}_{M,q-1}:\mathcal{L}^{2}_{(0,q-1)}(M) \rightarrow
\mathcal{L}^{2}_{(0,q)}(M)$ (and the range is interpreted as $\{0\}$ when $q=0$).
$S_{q}^{\prime\prime}$ denotes the orthogonal projection $\mathcal{L}^{2}_{(0,q)}(M)
\rightarrow Im((\overline{\partial}_{M,q})^{*})$ (when $q=(m-1)$, this range is
interpreted as $\{0\}$). Note that the three projections provide an orthogonal
decomposition of $\mathcal{L}^{2}_{(0,q)}(M)$, and $u = S_{q}^{\prime}u +
S_{q}^{\prime\prime}u + H_{q}u$, $0\leq q \leq (m-1)$.

The 1-form $\alpha$ mentioned in the introduction is defined as follows. Let $T$ be the
vector field introduced at the beginning of this section. Denote by
$\eta$ the purely
imaginary $1$--form on $M$ dual to $T$ (i.e. $\eta(T) \equiv 1$, and $\eta$ vanishes on
$T^{1,0}(M) \oplus T^{0,1}(M)$). Then $\alpha$ is the negative of the Lie derivative of $\eta$ in the direction of $T$:
\begin{equation}\label{alpha}
\alpha := -\{Lie\}_{T}(\eta)\;.
\end{equation}
Note that $\alpha$ is real. $\alpha$ is important because it arises when expressing $T$--components of certain commutators (\cite{D'Angelo93}, p.92, \cite{BoasStraube93}).
Indeed, when $X \in T^{1,0}(M)$, the definition of the Lie derivative (see for example \cite{GHV72}, section 4.3) gives
\begin{equation}\label{alphacomp}
\alpha\left(X\right) = -\left(T\eta(X) - \eta([T,X])\right ) = \eta([T,X]) \;
\end{equation}
(since $\eta(X)\equiv 0$ on $M$ and $T$ is tangent to $M$). Thus $\alpha(X)$ is the $T$--component of the commutator $[T, X]$.

The form $\alpha$ was introduced into the literature by D'Angelo (\cite{D'Angelo80, D'Angelo86, D'Angelo93}), also for the purpose of dealing with commutators as in \eqref{alphacomp}. Its role in the context of estimates for the $\overline{\partial}$--Neumann operator was discovered in \cite{BoasStraube93}. A detailed discussion of this role may be found in \cite{Straube10a}, sections 5.9--5.12.

Denote by $K \subset M$ the set of weakly
pseudoconvex points of $M$, and by $\mathcal{N}_{z}$ the null space of the Levi form at
the point $z \in K$. We say that \emph{$\alpha$ is exact on the null space
of the Levi form} if there exists a smooth function $h$,
defined in a neighborhood of $K$ (in $M$), such that
\begin{equation}\label{exact-null}
dh(L_{z})(z) = \alpha(L_{z})(z) \;, \; L_{z} \in
\mathcal{N}_{z},\;z \in K \, .
\end{equation}
We are now ready to state our results.

\begin{theorem}\label{main}
Let $M$ be a smooth compact pseudoconvex orientable CR submanifold of $\mathbb{C}^{n}$ of
hypersurface type, of CR--dimension (m-1). Assume that $\alpha = \alpha_{M}$ is
exact on the null space of the Levi form. Then for every nonnegative real number $s$,
there is a constant $C_{s}$ such that for all $u \in \mathcal{L}^{2}_{(0,q)}(M)$, $0\leq
q\leq (m-1)$,
\begin{equation}\label{main1}
\|S_{q}^{\prime}u\|_{s} + \|S_{q}^{\prime\prime}u\|_{s} +
\|H_{q}u\|_{s} \leq C_{s}\|u\|_{s}\;;\hskip 2.3in
\end{equation}
\begin{equation}\label{main2}
\|u\|_{s} \leq C_{s}\left(\|\overline{\partial}_{M}u\|_{s} +
\|\overline{\partial}_{M}^{*}u\|_{s} + \|u\|\right)\;;\; 1\leq q\leq (m-2)\;;\hskip 0.98in
\end{equation}
\begin{equation}\label{main3}
\|u\|_{s} \leq C_{s}\left(\|\overline{\partial}_{M}u\|_{s} +
\|\overline{\partial}_{M}^{*}u\|_{s}\right)\;;\; u\perp \mathcal{H}_{q}\;; \hskip2.26in
\end{equation}
\begin{equation}\label{main4}
\|G_{q}u\|_{s} \leq C_{s}\|u\|_{s}\;. \hskip 3.58in
\end{equation}
\end{theorem}

\smallskip
Estimates \eqref{main1} and \eqref{main4} say, respectively, that the three projections $S_{q}^{\prime}$, $S_{q}^{\prime\prime}$, and
$H_{q}$, and the complex Green operators $G_{q}$ are continuous in Sobolev norms. So are
the canonical solution operators to $\overline{\partial}_{M}$ and to
$\overline{\partial}_{M}^{*}$, by \eqref{main3}. In \eqref{main2}, we can replace the
$\|u\|$ term on the right hand side by $\|H_{q}u\|$, in view of \eqref{basic}. As a result, \eqref{main3} follows trivially from \eqref{main2} and \eqref{basic} when $1\leq q\leq (m-2)$.
Thus the main cases of interest in \eqref{main3} are the cases $q=0$ and $q=(m-1)$.

Estimate \eqref{main2} immediately gives that harmonic forms are in $W^{s}_{(0,q)}(M)$ for all
$s\geq 0$, hence are smooth, when $1\leq q\leq (m-2)$:
\begin{corollary}\label{harmsmooth}
Let $M$ (and $\alpha_{M}$) satisfy the assumptions of Theorem \ref{main}. Then
\begin{equation}\label{harmsmooth2}
\mathcal{H}_{q}(M) \subset C^{\infty}_{(0,q)}(M)\;,\;1\leq q \leq (m-2)\;.
\end{equation}
\end{corollary}

We point out that there are manifolds $M$ as in Corollary \ref{harmsmooth} with $\mathcal{H}_{q}(M)\neq \{0\}$ (but finite dimensional) for say $q=1$. This can be seen as follows.

In \cite{Brink02}, the author noticed that there exist smooth compact strictly pseudoconvex submanifolds of hypersurface type and of any dimension, embedded into some $\mathbb{C}^{n}$, whose (smooth) $\overline{\partial}_{M}$--cohomology at the level of $(0,1)$--forms is nontrivial. Examples arise in \cite{CatLem92}, where the authors construct smooth strictly pseudoconvex compact embeddable CR manifolds
which admit small deformations that are also embeddable, but whose embeddings cannot be chosen close to the original embedding. A theorem of Tanaka (\cite{Tan75}) shows that for such a manifold $M$ (with CR--dimension at least two) the $\overline{\partial}_{M}$--cohomology at the level of $(0,1)$--forms cannot be trivial (if it were, the small deformations would have to be embeddable by embeddings close to the original one). These observations are contained in the Remark at the end of \cite{Brink02}.

The manifolds from the previous paragraph are also orientable, as they arise as boundaries of complex manifolds (\cite{CatLem92}, page 103). Moreover, they trivially satisfy the assumptions of Theorem \ref{main}: any $h\in C^{\infty}(M)$ will do in \eqref{exact-null}, as $\mathcal{N}_{z}=\{0\}$ for all $z\in M$. Therefore, if $\beta$ is a smooth $\overline{\partial}_{M}$--closed $(0,1)$--form such that $\overline{\partial}_{M}f=\beta$ admits no smooth solution $f$, this equation also does not admit a solution in $\mathcal{L}^{2}(M)$. If it did, the canonical solution would have to be smooth, in view of \eqref{main3} in Theorem \ref{main}. Thus for such $M$, $\mathcal{H}_{1}(M)\neq \{0\}$. 

\medskip
We give two classes of CR submanifolds that satisfy the assumptions of Theorem
\ref{main}. For the first, observe that because $M$ is orientable, there is a (tubular)
neighborhood $V$ of $M$ so that within $V$, $M$ is given globally by defining functions
$\rho_{1}, \cdots, \rho_{l}$: $M = \{z \in V|\rho_{j}=0,\, j=1,\cdots,l\}$. here $l$ is
the real codimension of $M$, that is, $l=2n-(2m-1)$. Theorem \ref{plush} is the analogue
of the main result in \cite{BoasStraube91} obtained for the case where $M$ is
the (smooth) boundary of a bounded pseudoconvex domain in $\mathbb{C}^{n}$.\footnote{In
\cite{BoasStraube91}, the authors
only needed to assume that the defining function is plurisubharmonic at points of the
boundary. Our proof does use plurisubharmonicity in some
(arbitrarily small) neighborhood of $M$. This may be an artifact of the proof. On
the other hand, in terms of actually verifying the assumption, not much is lost. The role
of $\alpha_{M}$ is not made explicit in \cite{BoasStraube91}.}

\begin{theorem}\label{plush}
Let $M$ be a smooth compact pseudoconvex orientable CR submanifold of $\mathbb{C}^{n}$ of
hypersurface type. Assume that $M$ admits a set of
plurisubharmonic defining functions in some neighborhood. Then $\alpha_{M}$ is
exact on the null space of the Levi form. Consequently, the conclusions of
Theorem \ref{main} and Corollary \ref{harmsmooth} hold.
\end{theorem}

The assumption in Theorem \ref{plush} relies heavily on $M$ being embedded in
$\mathbb{C}^{n}$. Our next result has a more intrinsic flavor. It is analogous to
results for the $\overline{\partial}$--Neumann operators in \cite{BoasStraube93}.

\begin{theorem}\label{submanifold}
Let $M$ be a smooth compact pseudoconvex orientable CR submanifold of $\mathbb{C}^{n}$ of
hypersurface type, strictly pseudoconvex except for the points of a closed smooth submanifold $S$ (with or without boundary). Suppose that at each
point of $S$, the (real) tangent space is contained in the null space of the Levi form
(of $M$) at the point. If the first DeRham cohomology $H^{1}(S)$ is trivial, then $\alpha_{M}$
is exact on the null space of the Levi form. Consequently, the conclusions of
Theorem \ref{main} and Corollary \ref{harmsmooth} hold.
\end{theorem}

For emphasis, we formulate the following important special case as a corollary.

\begin{corollary}\label{complex-sub}
Let $M$ be a smooth compact pseudoconvex orientable CR submanifold of $\mathbb{C}^{n}$ of
hypersurface type, strictly pseudoconvex except for a simply connected complex
submanifold (smooth as a submanifold with boundary). Then $\alpha_{M}$ is
exact on the null space of the Levi form. Consequently, the conclusions of
Theorem \ref{main} and Corollary \ref{harmsmooth} hold.
\end{corollary}

The conclusions in Theorem \ref{submanifold} and in Corollary
\ref{complex-sub} suggest that $\alpha$ restricted to the submanifold $S$ be closed. That
this is indeed the case was discovered in \cite{BoasStraube93}. This fact is crucial for the proof of
Theorem \ref{submanifold}. In particular, $\alpha$ represents a DeRham cohomology class $[\alpha]$ on $S$, and the assumptions in Theorem \ref{submanifold} and Corollary \ref{complex-sub} imply that $[\alpha] = 0$. The appearance of this cohomology class explains why an annulus may or may not be an obstruction to Sobolev estimates (wormlike vs. non-wormlike Hartogs domains \cite{BoasStraube92}), and why a disc is always benign (\cite{BoasStraube92, BoasStraube93}).

\bigskip
\emph{Remark 1}: The estimates in Theorem \ref{main} are not independent. For example, in
\cite{HPR13} the authors show, among
many other things, that regularity of $G_{q}$ is equivalent to regularity of three
Szeg\"{o} type projections at levels $(q-1)$, $q$, and $(q+1)$ (for the
$\overline{\partial}$--Neumann operators and the Bergman
projections, this was shown in \cite{BoasStraube90}). It was observed already in \cite{BoasStraube91} that \eqref{main4} is an easy consequence of \eqref{main3} and \eqref{main1}; see also \eqref{Green} below.


\medskip
\emph{Remark 2}: The method used to prove Theorem \ref{submanifold} can be adapted to
also obtain results when $M$ contains a Levi flat patch that is foliated by complex
manifolds of dimension $(m-1)$ (the so called Levi foliation of the patch). This
is of some intrinsic interest; the solvability of $dh = \alpha$ on the whole patch (i.e.
\eqref{exact-null} above) turns out to be equivalent to a question that is much studied in
foliation theory, namely whether the Levi foliation can be given \emph{globally} by a
\emph{closed} $1$--form. We refer the reader to the discussion in section 5.11 in
\cite{Straube10a} and the references given there.


\medskip
\emph{Remark 3}: It is not known whether the converse to Corollary \ref{complex-sub}
holds. That is, if say $G_{1}$ satisfies Sobolev estimates, does it follow that the
restriction of $\alpha$ to $S$ is exact (i.e. $[\alpha]=0$)? This
question is also open in the context of the $\overline{\partial}$--Neumann problem. For a
partial result, see \cite{Barrett94}.

\section{Proof of Theorem \ref{plush}}\label{plush-proof}

Denote by $J$ the usual complex structure map induced by the complex structure of
$\mathbb{C}^{n}$. Suppose $M$ is defined by the defining functions $\rho_{1}, \cdots,
\rho_{l}$ which are plurisubharmonic in some neighborhood of $M$. Then $l = 2n -2m + 1$,
and the real gradients are linearly independent over $\mathbb{R}$ at points of $M$.
Because $T^{\mathbb{C}}(M)$ is $J$-invariant, and because $J$ preserves inner products,
$JT$ is orthogonal to $M$.

By Theorem 2.2 in \cite{Baracco1}, $M$ has a one-sided complexification to a complex
submanifold $\widehat{M}$ of $\mathbb{C}^{n}$ (a `strip'), so that $M$ is the smooth
connected component of the boundary of $\widehat{M}$ from the pseudoconvex side. Denote
by $\tilde{T}$ the real unit normal to $M$ pointing `outside' $\widehat{M}$, that is,
$\tilde{T} = -iJT$. The following geometric Lemma contains the crux of the matter.

\begin{lemma}\label{CR-geometry}
Let $U$ be a neighborhood of $M$. Then,
near $M$, $\widehat{M}$ is contained in the hull of $M$ with respect to the functions that are plurisubharmonic in $U$.
\end{lemma}

\begin{proof}
The proof is more or less implicit in the proof of Theorem 2.2 in \cite{Baracco1}. One
has to observe that the extensions there can be swept out by analytic discs with
boundaries in $M$, or in a set already under control. This does necessitate a
modification, as \cite{Baracco1} at one point uses propagation of extendibility from
\cite{HangesTreves83}, which is not based on analytic discs. 

First note that there are strictly pseudoconvex points on $M$. This can be seen by enclosing $M$ inside a
large sphere and then shrinking the radius. Points of first contact with $M$ are strictly pseudoconvex points of $M$. Near a strictly pseudoconvex point $z_{0}$ of $M$, $\widehat{M}$ is constructed as follows.
Graph $M$, near $z_{0}$, over its projection $\pi$ into a suitable copy of
$\mathbb{C}^{m}$. Then the hypersurface $\pi(M)$ is strictly pseudoconvex at $\pi(z_{0})$, and the
inverse of the projection extends to the pseudoconvex side of $\pi(M)$ (as all
components
are CR functions on $\pi(M)$), by the Kneser--Lewy extension theorem\footnote{Often referred to as the Lewy extension theorem, but see \cite{Kneser36, Range02}.}.
Moreover, near $\pi(z_{0})$, the pseudoconvex side of $\pi(M)$ can be
filled in by analytic discs
with boundaries in $\pi(M)$, and these lift to analytic discs with boundaries in $M$ that sweep out the extension. In particular, by the maximum principle for plurisubharmonic functions, near $z_{0}$ the extension constructed in this manner is contained (after shrinking if necessary) in the hull of $M$ with respect to the plurisubharmonic functions in $U$.

Next, note that $M$ consists of a single CR orbit (\cite{Baracco1}, Proposition 2.1). Pick a strictly pseudoconvex point $z_{0}\in M$. Let $z_{1}$ be an arbitrary point of $M$, and let $\gamma(t)$, $0\leq t\leq 1$ be a (piecewise smooth) CR curve with $\gamma(0)=z_{0}$ and $\gamma(1)=z_{1}$. Let $A$ be the set of those $t\in [0,1]$ with the property that near $\gamma(t)$, $\widehat{M}$ is contained in he hull of $M$ with respect to the plurisubharmonic functions in $U$. By what was said above , $0\in A$, i.e. $A$ is not empty. By definition, $A$ is open. It now suffices to show that $A$ is also closed; then $A=[0,1]$, $1\in A$, and  $\widehat{M}$ is contained in the required hull near $z_{1}$.

First assume that $t_{0}\in \overline{A}$ is such that  $\gamma(t_{0})\in M$ is not contained in any (germ of a) complex submanifold of $M$ of dimension $(m-1)$. Near $\gamma(t_{0})$, we can again graph $\widehat{M}$ over its projection $\pi$ into a suitable copy of $\mathbb{C}^{m}$. Then $\pi(M)$ is a hypersurface, and there is no germ of a complex submanifold of $\pi(M)$ of dimension $(m-1)$ that contains $\pi(\gamma(t_{0}))$. It follows that $\pi(M)$ is minimal at $\pi(\gamma(t_{0}))$, in the terminology of \cite{BER99}. Also, there are points arbitrarily close to $\pi(\gamma(t_{0}))$ where the Levi form has at least one positive eigenvalue (otherwise a neighborhood of $\pi(\gamma(t_{0}))$ would be foliated by complex submanifolds of $\pi(M)$ of dimension $(m-1)$). In particular, it is clear which side of $\pi(M)$ is the pseudoconvex side. Because $\pi(M)$ is minimal at $\pi(\gamma(t_{0}))$, there is a one--sided neighborhood of $\pi(\gamma(t_{0}))$ on the pseudoconvex side of $\pi(M)$ that is swept out 
by analytic discs with boundaries in $\pi(M)$. Moreover, these discs can be chosen `small': for any neighborhood of $\pi(\gamma(t_{0}))$, the construction can be done within that neighborhood.
This follows from \cite{BER99}, Theorem 8.6.2 and the proof of Theorem 8.6.1. there. It is not explicitly stated in \cite{BER99} that the swept out one--sided neighborhood is on the pseudoconvex side when the hypersurface is pseudoconvex, but this property follows for example from the standard characterization of pseudoconvexity via families of analytic discs\footnote{Via the continuity principle (see e.g.\cite{Range86}, Theorem 5.8 in section 5.4). Indeed, if there were a (small) disc with boundary in $\pi(M)$ and non-empty intersection with the pseudoconcave side of $\pi(M)$, translating it along the normal to $\pi(M)$ at $\pi(\gamma(t_{0}))$ would produce a one parameter family of discs that contradicts the continuity principle on the pseudoconvex side of $\pi(M)$.}. Lifting these discs via the graphing function(s) for $\widehat{M}$ gives a family of analytic discs with boundaries in $M$ that sweeps out a neighborhood of $z_{0}$ in $\widehat{M} \cup M$. This neighborhood is contained in the hull of $M$ 
with respect to the functions that are plurisubharmonic in $U$. (This part of the argument does not use the fact that $t_{0}\in \overline{A}$.)

Assume now that $t_{0} \in \overline{A}$ and $\gamma(t_{0})\in S \subset M$, where $S$ is a (germ of) an $(m-1)$ dimensional complex submanifold of $M$. In this case, we use results from \cite{Tumanov95} (see also \cite{Tumanov94}). With $\pi$ a projection as above, we have $\pi(\gamma(t_{0})) \in \pi(S) \subset \pi(M)$, and $\pi(S)$ is an $(m-1)$ dimensional complex submanifold of the hypersurface $\pi(M)$. For $t$ close to $t_{0}$, $\pi(\gamma(t))$ gives a CR curve in $\pi(M)$. Near $\pi(\gamma(t_{0}))$, it must stay inside $\pi(S)$. Choose a point $w$ on it that is the projection of a point $\gamma(t_{1})$ with $t_{1} \in A$. Let $M^{\prime}$ be the projection of a (small) neighborhood of $\gamma(t_{1})$ in $\widehat{M}\cup M$. $M^{\prime}$ is a one--sided neighborhood of $w$, and its boundary coincides with $\pi(M)$ near $w$. The proof of Theorem 3.3 in \cite{Tumanov95} shows the following. There are a finite sequence of points $q_{j}$, $1\leq j\leq k$ on $\pi(M)$ (close to the CR curve $\pi(\gamma)$) 
and one--sided neighborhoods $M_{j}$ of $q_{j}$, $1\leq j\leq k$, with the following properties: (i) $q_{1}$ is close to $w$ and $M_{1} \subseteq M^{\prime}$, (ii) $q_{k}=\pi(\gamma(t_{0}))$,
and (iii) $M_{j+1}$ is swept out by analytic discs with boundaries in $\pi(M)\cup M_{j}$, $1\leq j\leq (k-1)$. Moreover, this can be done so that the $M_{j}$'s as well as all the discs involved stay in as small a neighborhood of $\pi(\gamma(t_{0})))$ as we wish. Again by the characterization of pseudoconvexity via families of analytic discs, all the discs involved stay on the same side of $\pi(M)$ (the pseudoconvex side\footnote{$\pi(M)$ may be Levi flat near $\pi(\gamma(t))$, so that both sides are pseudoconvex. However, the pseudoconvex side of $M$ is defined globally (it is given by $iJT$). The local projections $\pi$ near a point in $M$ then transfer this direction/side ``downstairs''.}). Lifting the neighborhoods $M_{j}$ and the discs involved via the graphing function(s) of $\widehat{M}$ (the inverse of $\pi$) and applying the maximum principle for plurisubharmonic functions shows that there is indeed a neighborhood of $\gamma(t_{0})$ in $M\cup\widehat{M}$ that is contained in the hull of $M$ with 
respect to the 
plurisubharmonic functions in $U$ (namely the image of $M_{k}$). In other words, $t_{0}\in A$.

Since $z_{1}$ was an arbitrary point of $M$, the proof of Lemma \ref{CR-geometry} is complete.
\end{proof}

Set $\rho := \rho_{1} + \cdots + \rho_{l}$. The restriction of $\rho$ to $\widehat{M}$ can serve as a one--sided defining function for $M$ (on $\widehat{M}$), in view of Lemma \ref{nonzero}:

\begin{lemma}\label{nonzero}
$\tilde{T}\rho > 0$ on $M$.
\end{lemma}

\begin{proof}
Let $z \in M$. Then there is at least one index $j_{0}$ such that
$\nabla_{\mathbb{R}}\rho_{j_{0}}(z)$ is not orthogonal to $\tilde{T}$. This is because the
(real) gradients of the defining functions span (over $\mathbb{R}$) the orthogonal
complement of $M$. In particular, $\tilde{T}\rho_{j_{0}}(z) \neq 0$. Also,
in view of Lemma \ref{CR-geometry}, $\rho_{j} \geq 0$ on $\widehat{M}$, $1 \leq j \leq l$.
Therefore, $\tilde{T}\rho_{j}(z) \geq 0$ for all $j$, $1 \leq j \leq l$. Consequently,
$\tilde{T}\rho_{j_{0}}(z) > 0$, and $\tilde{T}\rho(z) =
\tilde{T}(\rho_{1} + \cdots + \rho_{l})(z) > 0$.
The proof of Lemma \ref{nonzero} is complete.
\end{proof}

We are now ready to prove Theorem \ref{plush}; the essence of the argument/computation is
the same as in \cite{BoasStraube91a, BoasStraube93, Straube10}, but the organization is
somewhat different.

\smallskip
\begin{proof}[Proof of Theorem \ref{plush}]
Denote by $J^{*}$ the adjoint of $J$ with respect to the pairing between vector fields
and forms (\cite[page 42]{Boggess91}). Recall that $J^{*}\partial =
i\partial$ and $J^{*}\overline{\partial} = -i\overline{\partial}$. In particular,
\begin{equation}\label{comp1}
\left (\partial\rho - \overline{\partial}\rho\right )(T) = -\left (\partial\rho -
\overline{\partial}\rho\right )(J^{2}T) = -\left (J^{*}\partial\rho -
J^{*}\overline{\partial}\rho\right )(JT) = -id\rho(JT) > 0 \; ;
\end{equation}
by Lemma \ref{nonzero}. Therefore, $\eta = e^{h}\left (\partial\rho -
\overline{\partial}\rho\right )$ for some $h \in C^{\infty}(M)$ (since also $\left
(\partial\rho -
\overline{\partial}\rho\right )$ vanishes on $T^{1,0}(M) \oplus T^{0,1}(M)$). We claim
that
\begin{equation}\label{claim}
\alpha\left(\overline{L}\right )(z) = dh\left(\overline{L}\right)(z)\;,\;z \in M\;,\;L \in
\mathcal{N}_{z} \;.
\end{equation}

Using one of the Cartan formulas for the Lie derivative (see for example \cite{GHV72}, section 4.3, Proposition II, (1))
and the fact that $\eta(T)\equiv 1$
gives
\begin{equation}\label{alphacalc2}
\alpha\left (\overline{L}\right ) =
-\left((i(T)d + d\,i(T))\eta \right)(\overline{L}) = -d\eta(T, \overline{L}) \;,
\end{equation}
where $i(T)$ denotes the substitution operator (\cite{GHV72}, section 4.1).
The fact that $\eta =
e^{h}\left(\partial\rho-\overline{\partial}\rho\right)$ implies
\begin{equation}\label{alphacalc3}
d\eta = e^{h}\left[dh\wedge\left(\partial\rho-\overline{\partial}\rho\right)
+ d\left(\partial\rho-\overline{\partial}\rho\right) \right] = dh\wedge\eta -
2e^{h}\partial\overline{\partial}\rho~.
\end{equation}
We have used again that $e^{h}\left (\partial\rho -
\overline{\partial}\rho\right )=\eta$, and that $d = \partial + \overline{\partial}$. Now
insert \eqref{alphacalc3} into the last term of \eqref{alphacalc2} and observe that
$\eta\left(\overline{L}\right) =0$ implies that
$\left(dh\wedge\eta\right)\left(T, \overline{L}\right) =
-dh\left(\overline{L}\right)\eta(T) = -dh\left(\overline{L}\right)$.
The result is
\begin{equation}\label{alphacalc4}
\alpha\left(\overline{L}\right)(z) = dh\left( \overline{L}\right)(z) +
2e^{h}\partial\overline{\partial}\rho\left(T, \overline{L}\right)(z)\;, \; z \in
M\;,\;L \in \mathcal{N}_{z}\;.
\end{equation}
The plurisubharmonicity of $\rho$ says that $i\partial\overline{\partial}\rho$ is
positive semi--definite. Therefore, $L \in \mathcal{N}_{z}$ implies that the second term
on the right hand side of \eqref{alphacalc4} vanishes (note that
$\partial\overline{\partial}\rho\left(T, \overline{L}\right)=
i\partial\overline{\partial}\rho\left(-iT, \overline{L}\right) =
i\partial\overline{\partial}\rho\left((-iT)_{(1,0)}, \overline{L}\right)$, where the
subscript denotes the $(1,0)$--part of the vector $-iT$). Thus
$\alpha(\overline{L})(z)=dh(\overline{L})(z)$. Since $\alpha$ is real, we also have
$\alpha(L)(z)=dh(L)(z)$. We have shown that $\alpha$ agrees with $dh$, and so is
exact, on the null space of the Levi form. The proof of Theorem \ref{plush} is complete.
\end{proof}

\medskip
\emph{Remark 4}: It is worthwhile to note that the particular combination $\rho = \rho_{1} + \cdots +
\rho_{l}$ does not matter in Lemma \ref{nonzero} (and the rest of the argument) above; any
combination $c_{1}\rho_{1} + \cdots + c_{l}\rho_{l}$ with $c_{j}>0$, $1\leq j \leq l$,
will do. Geometrically, this says the following. For a point $z$ of $M$, consider the
positive cone $C_{z}$ in $\mathbb{C}^{n} \approx \mathbb{R}^{2n}$ generated by the
gradients of the $\rho_j$'s,
$C_{z} = \{c_{1}\nabla\rho_{1} + \cdots + c_{l}\nabla\rho_{l}\,| \, c_{j}>0\,,\,1\leq
j\leq l\}$. Denote by $\widehat{C_{z}}$ its dual cone. Then the extension
$\widehat{M}$ goes in the direction of $-\tilde{T}$, and so near $M$ is contained in the union $\cup_{z\in M}(z+\widehat{C_{z}})$ of dual cones.


\section{Proof of Theorem \ref{submanifold}}\label{submanifold-proof}

The proof is essentially from \cite{BoasStraube93} (see also \cite{Straube10a}, proof of Corollary 5.16), where a slightly weaker conclusion is obtained when $M$ is the boundary of a smooth bounded pseudoconvex domain.
\begin{proof}[Proof of Theorem 3]
 The crucial fact for the proof, discovered in \cite{BoasStraube93} (see the Lemma on p.230; see also \cite{Straube10a}, Lemma 5.14), is that $\alpha$ is closed on the null space of the Levi form, that is , $d\alpha|\mathcal{N}_{z} = 0$, for $z \in M$. Consequently, $\alpha$ represents a DeRham cohomology class on $S$. Because $H^{1}(S)$ is trivial, $\alpha = d\tilde{h}$ on $S$, for some $C^{\infty}$--function $\tilde{h}$ on $S$ (if $S$ is a submanifold with boundary, then $\tilde{h}$ is smooth up to the boundary). Locally, $\tilde{h}$ can be extended into a neighborhood (in $M$) of $S$ in such a way that \emph{for points on $S$}, the differential of the extended function, at the point, agrees with $\alpha$ at the point. This is also true near points of the boundary, if $S$ is a submanifold with boundary. By compactness, we can choose a finite cover $\{U_{j}\}_{j=1}^{J}$ of $S$, associated extensions $\{h_{j}\}_{j=1}^{J}$ of $\tilde{h}$, and a partition of unity $\{\varphi_{j}\}_{j=1}^{J}$ of a neighborhood
of $S$ that is subordinate to this cover. We now set $h:= \sum_{j=1}^{J}\varphi_{j}h_{j}$. The function $h$ is defined in a neighborhood of $S$, and for $z \in S$,
 \begin{multline}
  \;\;\;\;\;\;\;\;\;\;\;dh(z) = \sum_{j=1}^{J}\varphi_{j}(z)dh_{j}(z) + \sum_{j=1}^{J}d\varphi_{j}(z)h_{j}(z) \\
  = \sum_{j=1}^{J}\varphi_{j}(z)\alpha(z) + \sum_{j=1}^{J}d\varphi_{j}(z)\tilde{h}(z) = \alpha(z) \; .\;\;\;\;\;\;\;
 \end{multline}
We have used here that $\sum_{j=1}^{J}d\varphi_{j}(z) = d(\sum_{j=1}^{J}\varphi_{j})(z) = 0$. Using a cutoff function supported in a small enough neighborhood of $S$ and identically one in a (smaller) neighborhood, we can extend $h$ to a $C^{\infty}$--function on $M$. This concludes the proof of Theorem \ref{submanifold}.

 \end{proof}


\section{Exactness of $\alpha$ and good vector fields}\label{alpha-fields}

Recall $K\subset M$ denotes the set of weakly pseudoconvex points of $M$ and
$\mathcal{N}_z$
denotes the null space of the Levi form at point $z\in K$.  When proving Sobolev
estimates, one needs vector fields with good commutator properties. The required
properties come for free for commutators with vector fields in strictly pseudoconvex
directions. This crucial observation was made in \cite{BoasStraube91a} (in the context
of the $\overline{\partial}$--Neumann problem), see also \cite{Straube10}, section 5.7.
For commutators with fields in weakly pseudoconvex directions, the needed
commutator properties come from the exactness of $\alpha$ on the null space of the Levi
form.

When $X$ is a vector field on $M$, let us denote its $T$--component modulo $T^{1,0}(M)\oplus T^{0,1}(M)$ by $(X)_{T}$.

\begin{proposition}\label{commutators}
Let $M$ be a smooth compact pseudoconvex orientable CR submanifold of $\mathbb{C}^{n}$ of
hypersurface type. Assume that $\alpha_{M}$ is
exact on the null space of the Levi form. Then there exists a constant $C$ such that for every  $\varepsilon>0$, there are 
smooth real vector fields $X_{\varepsilon}$ on $M$ with the following properties:

(i)
\begin{equation}\label{uniformly transversal}
1/C \leq  \left|(X_{\varepsilon})_{T}\right|   \leq C \; ,
\end{equation}
and

(ii)
\begin{equation}\label{uniformlysmall}
\left|([X_{\varepsilon},\overline{Z}])_{T}\right| \leq \varepsilon
\end{equation}
for every unit vector field $Z$ in $T^{1,0}(M)$.
\end{proposition}


\begin{proof}[Proof of Proposition \ref{commutators}]
The proof consists in combining, and rewriting in the present context, the arguments from
the proofs of Theorem 5.9, Lemma 5.10, and Proposition 5.13 in \cite{Straube10a}, or from
\cite{StraubeSucheston02}, equivalence of (iii) and (iv) in the theorem there.

Fix $\varepsilon > 0$. Locally, near a point $P\in M$, the vector field
$X_{\varepsilon}$ we seek can be written as

\begin{equation}\label{X-epsilon}
X_{\varepsilon} = e^{g_{\varepsilon}}T + \sum_{j=1}^{m-1}\left(b_{\varepsilon,j}L_{j} +
\overline{b_{\varepsilon,j}L_{j}}\right) \; ,
\end{equation}
for smooth functions $g_{\varepsilon}$ and $b_{\varepsilon,j}$ that are to be determined.
Here, $L_{1}, \cdots, L_{m-1}$ generate a basis of $T^{\mathbb{C}}(M)$ near $P$. We have used
that $X_{\varepsilon}$ is real. We may choose $L_{1}, \cdots,
L_{m-1}$ so that \emph{at} $P$ (but not necessarily near $P$), this basis diagonalizes the
Levi form. Then
we have, at $P$,
\begin{multline}\label{T-component}
 \left[X_{\varepsilon}, \overline{L_{k}}\right]_{T}(P) =
-e^{g_{\varepsilon}(P)}\overline{L_{k}}g_{\varepsilon}(P) +
e^{g_{\varepsilon}(P)}\left[T,\overline{L_{k}}\right]_{T}(P) + \sum_{j=1}^{m-1}
b_{\varepsilon,j}(P)\left[L_{j}, \overline{L_{k}}\right]_{T}(P) \\
= -e^{g_{\varepsilon}(P)}\left(\overline{L_{k}}g_{\varepsilon}(P) -
\alpha(\overline{L_{k}})(P)\right) +
b_{\varepsilon,k}(P)\left[L_{k},\overline{L_{k}}\right]_{T}(P) \; .
\end{multline}
For the first equality, note that all the other terms in the commutators coming from the
sum in \eqref{X-epsilon} are multiples of one of the $L_{j}$, $\overline{L_{j}}$, or
$\left[\overline{L_{j}},\overline{L_{k}}\right]$, and so are complex tangential, hence
have vanishing $T$--component. In the second equality, we have used \eqref{alphacomp},
which says that $[T,\overline{L_{k}}]_{T} = \alpha(\overline{L_{k}})$, and also the fact
that $[L_{j},\overline{L_{k}}]_{T} = \delta_{j,k}$ (because $L_{1}, \cdots, L_{m-1}$
diagonalizes the Levi form at $P$).

We first consider the case where $L_{k}(P) \in \mathcal{N}_{P}$. Then
$[L_{k},\overline{L_{k}}](P) = 0$, so that the relevant $T$--component becomes
$-e^{g_{\varepsilon}(P)}\left(\overline{L_{k}}g_{\varepsilon}(P) -
\alpha(\overline{L_{K}})(P)\right)$. It is now clear how the exactness of $\alpha$ on
$\mathcal{N}_{P}$ enters: if we choose $g_{\varepsilon}=h$, where $h$ is the function
from the definition of exactness of $\alpha$ on the null space of the Levi form (which we
may assume to be defined on all of $M$, rather than just in a neighborhood of $K$), then
$\overline{L_{k}}g_{\varepsilon}(P) - \alpha(\overline{L_{K}})(P) = 0$. Note that with
this choice, we also satisfy the requirement that the $T$--component of $X_{\varepsilon}$
be bounded and bounded away from zero uniformly in $\varepsilon$.

Next, consider the case where $L_{k} \notin \mathcal{N}_{P}$,
i.e. $[L_{k},\overline{L_{k}}]_{T}(P) \neq 0$. We are stuck with the contribution from the
first term on the right hand side of \eqref{T-component}; there is no reason why it should
be small. But now the term containing $[{L_{k}},\overline{L_{k}}]_{T}(P)$ comes to the
rescue. Indeed, if we choose the constant $b_{\varepsilon,k}$ as
\begin{equation}\label{b-epsilon}
b_{\varepsilon,k} = \frac{e^{h(P)}\left(\overline{L_{k}}g_{\varepsilon}(P) -
\alpha(\overline{L_{K}})(P)\right)}{ [L_{k},\overline{L_{k}}]_{T}(P)} \; ,
\end{equation}
then the right hand side of \eqref{T-component} vanishes.

The conclusion is that when we define $X_{\varepsilon,P} := e^{h}T +
\sum_{L_{j}\notin \mathcal{N}_{P}}\left(b_{\varepsilon,j}L_{j} +
\overline{b_{\varepsilon,j}L_{j}}\right)$ then
$[X_{\varepsilon,P},\overline{L_{j}}]_{T}(P) = 0$ for $1\leq j\leq (m-1)$. By continuity,
there is a neighborhood $V_{\varepsilon}$ of $P$ such that for $z\in V_{\varepsilon}$, we
have $\left |[X_{\varepsilon,P},\overline{Z}]_{T}(z)\right | < \varepsilon$ for any
section $Z$ of $T^{1,0}(M)$ of unit length. Choose finitely many points $P_{1}, \cdots,
P_{r}$ such that the corresponding neighborhoods $V_{\varepsilon} =
V_{\varepsilon,P_{r}}$ cover $M$, and let $\{\phi_{1}, \cdots, \phi_{r}\}$ be a partition
of unity subordinate to this cover. We set
\begin{equation}\label{X-epsilon1}
X_{\varepsilon} = \sum_{j=1}^{r}\phi_{j} X_{\varepsilon,P_{j}} =  e^{h}T +
\sum_{j=1}^{r}\phi_{j}\left(Y_{\varepsilon,P_{j}} +
\overline{Y_{\varepsilon,P_{j}}}\right) \; ,
\end{equation}
where $Y_{\varepsilon,P_{j}} = \sum_{L_{s}\notin\mathcal{N}_{P_{s}}}
b_{\varepsilon,s}L_{s}$ (for simplicity of notation, we have not added subscripts to the
$L$'s and the $b$'s to indicate the dependence on the point $P$ (or $P_{s}$)).
Then, if $Z$ is a field of type $(1,0)$ on $M$ of unit length,
\begin{equation}\label{finalcom}
\left[X_{\varepsilon},\overline{Z}\right]_{T} =
\sum_{j=1}^{r}\phi_{j}\left[X_{\varepsilon,P_{j}},\overline{Z}\right]_{T} -
\sum_{j=1}^{r}\overline{Z}\phi_{j}\left(X_{\varepsilon,P_{j}}\right)_{T} \; .
\end{equation}
Each term in the first sum on the right hand side of \eqref{finalcom} is at most
$\phi_{j}\varepsilon$ in absolute value. Therefore, the sum is no more than $\varepsilon$
in absolute value. In the second sum, note that $\left(X_{\varepsilon,P_{j}}\right)_{T} =
e^{h}$, independently of $j$, $1\leq j\leq r$. Combining this with
$\sum_{j=1}^{r}\overline{Z}\phi_{j} = \overline{Z}\left(\sum_{j=1}^{r}\phi_{j}\right) =
0$ gives that that sum vanishes. Consequently,
$\left[X_{\varepsilon},\overline{Z}\right]_{T}$ has modulus less that $\varepsilon$. Then
so does $\left[X_{\varepsilon},Z\right]_{T}$, the negative of its conjugate (recall that
$T$ is purely imaginary). The family of vector fields
$\{X_{\varepsilon}\}_{\varepsilon>0}$ has all the required properties, and the proof of
Proposition \ref{commutators} is complete.
\end{proof}


\section{Good vector fields and regularity: Proof of Theorem \ref{main}}\label{fields&reg}

The point is that a family of vector fields as in Proposition \ref{commutators} implies
the conclusion in Theorem \ref{main}. This is well known when $M$ is the
boundary of a pseudoconvex domain in $\mathbb{C}^{n}$ (\cite{BoasStraube91}; see in
particular Remark (2)), and it is clear that it should work in the present context.  Theorem \ref{main} is an immediate consequence of  Proposition \ref{commutators} and Theorem \ref{fields-estimates}.

\begin{theorem}\label{fields-estimates}
Let $M$ be a smooth compact orientable pseudoconvex CR submanifold of $\mathbb{C}^{n}$
of hypersurface type. Assume that for every  $\varepsilon>0$, there exists a
smooth real vector field $X_{\varepsilon}$ on $M$ whose $T$--component, mod $T^{1,0}(M)
\oplus T^{0,1}(M)$, is bounded and bounded away from zero, uniformly in $\varepsilon$,
with the property that
\begin{equation}\label{small-commutator1}
[X_{\varepsilon},\overline{Z}] = a_{\varepsilon}T \text{ mod }T^{1,0}(M)\oplus
T^{0,1}(M)
\end{equation}
for a function $a_{\varepsilon}$ of modulus less than $\varepsilon$  for every unit
vector field $Z$ in $T^{1,0}(M)$.
Then for every nonnegative real number $s$,
there is a constant $C_{s}$ such that for all $u \in \mathcal{L}^{2}_{(0,q)}(M)$, $0\leq
q\leq (m-1)$,
\begin{equation}\label{main1*}
\|S_{q}^{\prime}u\|_{s} + \|S_{q}^{\prime\prime}u\|_{s} +
\|H_{q}u\|_{s} \leq C_{s}\|u\|_{s}\;;\hskip 2.3in
\end{equation}
\begin{equation}\label{main2*}
\|u\|_{s} \leq C_{s}\left(\|\overline{\partial}_{M}u\|_{s} +
\|\overline{\partial}_{M}^{*}u\|_{s} + \|u\|\right)\;;\; 1\leq q\leq (m-2)\;;\hskip 0.98in
\end{equation}
\begin{equation}\label{main3*}
\|u\|_{s} \leq C_{s}\left(\|\overline{\partial}_{M}u\|_{s} +
\|\overline{\partial}_{M}^{*}u\|_{s}\right)\;;\; u\perp \mathcal{H}_{q}\;; \hskip2.26in
\end{equation}
\begin{equation}\label{main4*}
\|G_{q}u\|_{s} \leq C_{s}\|u\|_{s}\;. \hskip 3.58in
\end{equation}
\end{theorem}

\smallskip
\begin{proof}
The proof closely follows the arguments in \cite{BoasStraube91}. Some modifications are
needed, however. First, there are non trivial harmonic forms also when
$1\leq q \leq (m-2)$. Second, the case $m=2$ ($n=2$ in \cite{BoasStraube91}) needs extra
attention, as does the passage from a priori estimates to genuine estimates.

First, note that \eqref{main4*} is an easy consequence of \eqref{main1*} and
\eqref{main3*}. Indeed,
\begin{equation}\label{Green}
G_{q} = Q_{q}R_{q-1}S_{q}^{\prime} + R_{q}Q_{q+1}S_{q}^{\prime\prime}\;;\;0\leq q\leq
(m-1)\;,
\end{equation}
see equation (3) in \cite{BoasStraube91}. Here $R_{q}$ denotes the canonical solution operator on the range of
$\overline{\partial}_{M}:\mathcal{L}^{2}_{(0,q)}(M) \rightarrow
\mathcal{L}^{2}_{(0,q+1)}(M)$, with $R_{-1}$ and $R_{m-1}$ understood to be zero, and
$Q_{q}$ similarly denotes the canonical solution operator on the range of
$\overline{\partial}_{M}^{*}:\mathcal{L}^{2}_{(0,q)}(M) \rightarrow
\mathcal{L}^{2}_{(0,q-1)}(M)$, with $Q_{0}$ and $Q_{m-1}$ understood to be zero. \eqref{Green}
holds because the right hand side is zero on $\mathcal{H}_{q}^{\perp}$, maps into $\mathcal{H}_{q}^{\perp}$, and applying $\Box_{q}$ to it
results in $(Id - H_{q})$. \eqref{main3*} implies that the canonical solution operators are continuous in Sobolev norms. Consequently, \eqref{Green} displays $G_{q}$ as a sum of compositions of operators that are continuous in Sobolev norms.

By interpolation, it suffices to prove the
above estimates for integer values of $s$. This is straightforward for \eqref{main1*} and \eqref{main4*}; for \eqref{main2*} and \eqref{main3*}, one uses continuity of the canonical solution operators, obtained for example by combining the solution operators that result from the closed range property (\cite{Baracco1, Baracco2, Nicoara06, KohnNicoara06}) with the continuity of the projections in \eqref{main1*}.

The starting point is an estimate that says that complex tangential
derivatives are under control. Let $k$ be a positive integer, and $Y$ a smooth section of
$T^{1,0}(M) \oplus T^{0,1}(M)$. Then we have
\begin{multline}\label{BS-estimate}
\|Yu\|_{k-1}^{2} \leq C_{k}\left(\|\overline{\partial}_{M}u\|_{k-1}^{2}
+\|\overline{\partial}_{M}^{*}u\|_{k-1}^{2} + \|u\|_{k-1}\|u\|_{k}\right) \; , \\
u \in W^{k}_{(0,q)}(M) \;, \; 0 \leq q \leq (m-1) \;.
\end{multline}
When $M$ is the boundary of a pseudoconvex domain, \eqref{BS-estimate} is in \cite{BoasStraube91} (see Lemma 1 there). The
integration by parts argument used in the proof when $(k-1)=0$ works in the current
situation as well;
the general case then follows by applying the $(k-1)=0$ case to derivatives (as in
\cite{BoasStraube91}).

Fix $\varepsilon > 0$, to be chosen later. We can write derivatives of order $k$ in terms
of $X_{\varepsilon}$ and complex tangential derivatives.
When there is at least one complex tangential derivative, we can commute it to the right
(so it acts first), modulo an error that is of order $\|u\|_{k-1}$. Applying
\eqref{BS-estimate} to these terms, using $\|u\|_{k}\|u\|_{k-1} \leq
(sc)\|u\|_{k}^{2}+(lc)\|u\|_{k-1}^{2}$, using the interpolation inequality
$\|u\|_{k-1}^{2} \leq (sc)\|u\|_{k}^{2}+(lc)\|u\|^{2}$, and absorbing the $\|u\|_{k}^{2}$
term gives
\begin{multline}\label{BS-estimate1}
\|u\|_{k}^{2} \leq C\|X_{\varepsilon}^{k}u\|^{2}  + C_{\varepsilon}\left(
\|\overline{\partial}_{M}u\|_{k-1}^{2} +
\|\overline{\partial}_{M}^{*}u\|_{k-1}^{2} + \|u\|^{2}\right),\\
u\in W^{k}_{(0,q)}(M)\;,
0\leq q\leq (m-1)\;;
\end{multline}
here, the first constant does not depend on $\varepsilon$ (because the $T$--component of
$X_{\varepsilon}$ is bounded and bounded away from zero uniformly in $\varepsilon$).

We first prove \eqref{main2*} at the level of an a priori estimate. That is, we assume that $u$ is known to be in $W^{k}_{(0,q)}(M)$, and show that with this assumption, estimate \eqref{main2*} holds. Thus we assume that
$u$, $\overline{\partial}_{M}u$, and $\overline{\partial}_{M}^{*}u$ are in $W^{k}(M)$
(for the respective form levels); we want to show that \eqref{main2*} holds. The argument follows \cite[pages 1578-1579]{BoasStraube91}. Note that because
$\overline{\partial}_{M}X_{\varepsilon}^{k}u = [\overline{\partial}_{M}, X_{\varepsilon}^{k}]u +
X_{\varepsilon}^{k}\overline{\partial}_{M}u$, $X_{\varepsilon}^{k}u$ is in the domain of $\overline{\partial}_{M}$, and similarly, in the domain of $\overline{\partial}_{M}^{*}$.
Thus, when
$1\leq q \leq (m-2)$, we have for
the first term on the right hand side of \eqref{BS-estimate1}, in view of
\eqref{basic1},
\begin{multline}\label{eq2}
\|X_{\varepsilon}^{k}u\|^{2} \leq
C\left(\|\overline{\partial}_{M}X_{\varepsilon}^{k}u\|^{2} +
\|\overline{\partial}_{M}^{*}X_{\varepsilon}^{k}u\|^{2}\right) +
C_{\varepsilon}\|u\|_{k-1}^{2} \\
\leq C\left(\|[\overline{\partial}_{M}, X_{\varepsilon}^{k}]u\|^{2}
+ \|[\overline{\partial}_{M}^{*}, X_{\varepsilon}^{k}]u\|^{2}\right
) + C_{\varepsilon}\left(\|\overline{\partial}_{M}u\|_{k}^{2} +
\|\overline{\partial}_{M}^{*}u\|_{k}^{2} + \|u\|_{k-1}^{2}\right)\;;
\end{multline}
we have used that $\|X_{\varepsilon}^{k}u\|_{-1}^{2} \leq C_{\varepsilon}\|u\|_{k-1}^{2}$.
Assumption \eqref{small-commutator1} in Theorem \ref{fields-estimates} now lets us
estimate the commutators.
We use that $\left[\overline{\partial}_{M}, X_{\varepsilon}^{k}\right] =
X_{\varepsilon}^{k-1}\left[\overline{\partial}_{M}, X_{\varepsilon}\right] +$ terms
of order not exceeding $(k-1)$ (\cite{DerridjTartakoff76}, Lemma 2, p.418;
\cite{Straube10},
formula (3.54), p. 65). By \eqref{small-commutator1}, $\left[\overline{\partial}_{M},
X_{\varepsilon}\right] = a_{\varepsilon}T + Y_{\varepsilon}$, with $|a_{\varepsilon}| \leq
\varepsilon$ and $Y_{\varepsilon}$ complex tangential. Thus, in view of
\eqref{BS-estimate} and because the $T$--component of $X_{\varepsilon}$ is bounded
independently of $\varepsilon$,
\begin{multline}\label{eq1a}
\|[\overline{\partial}_{M}, X_{\varepsilon}^{k}]u\|^{2} \lesssim
\varepsilon^{2}\|T^{k}u\|^{2} + C_{\varepsilon}\left(\|Y_{\varepsilon}u\|_{k-1}^{2} +
\|u\|_{k-1}^{2}\right) \\
\lesssim \varepsilon^{2}\|u\|_{k}^{2} +
C_{\varepsilon}\left(\|\overline{\partial}_{M}u\|_{k-1}^{2}
+\|\overline{\partial}_{M}^{*}u\|_{k-1}^{2} + \|u\|_{k-1}\|u\|_{k}\right) \; .
\end{multline}
Inserting this estimate, and the analogous one for
$\|[\overline{\partial}_{M}^{*}, X_{\varepsilon}^{k}]u\|^{2}$, into \eqref{eq2}, and the
result into \eqref{BS-estimate1}, and using $\|u\|_{k-1}\|u\|_{k}  \leq (sc)\|u\|_{k}^{2}
+ (lc)\|u\|_{k-1}^{2} \leq (sc)\|u\|_{k}^{2} + (lc)\|u\|^{2}$, we obtain
\begin{equation}\label{eq5}
\|u\|_{k}^{2} \leq \varepsilon C\|u\|_{k}^{2} +
C_{\varepsilon}\left(\|\overline{\partial}_{M}u\|_{k}^{2} +
\|\overline{\partial}_{M}^{*}u\|_{k}^{2} + \|u\|^{2}\right)\;.
\end{equation}
Choosing $\varepsilon$ small enough and absorbing the $\|u\|_{k}^{2}$ term into the left
hand side of \eqref{eq5} establishes \eqref{main2*} (at the a priori level).

We now treat \eqref{main3*}, i.e. the cases $q=0$ and $q=(m-1)$ (we already mentioned
that when $1\leq q\leq (m-2)$, \eqref{main3*} follows trivially from \eqref{main2*},
in view of \eqref{basic}).First,
let $q=(m-1)$, and $u \in (\mathcal{H}_{m-1}(M))^{\perp} =
Im(\overline{\partial}_{M})$. There is a solution operator to
$\overline{\partial}$ that is continuous in both $W^{k}$ and $\mathcal{L}^{2}$--norm. When $m\geq 3$, the canonical
solution will satisfy this estimate by \eqref{main2*}, which is already established. When
$m=2$, such a solution operator is obtained in \cite{KohnNicoara06}, denoted by $\overline{\partial}_{b,t}^{*}N_{t}$ there, see the proof of Theorem 1.1, in particular Step 1 (p.264) (note that the range of $\overline{\partial}_{M}$ is closed in $\mathcal{L}^{2}$, as required in \cite{KohnNicoara06}, by \cite{Baracco1, Baracco2}). Therefore, we can write $u = \overline{\partial}_{M}v$, with $\|v\|_{k}\lesssim \|u\|_{k}$, $\|v\|\lesssim \|u\|$.
\eqref{basic} now gives (as above, the a priori assumptions imply that $X_{\varepsilon}^{k}v$ and $X_{\varepsilon}^{k}u$ are in the domains of $\overline{\partial}_{M}$ and $\overline{\partial}_{M}^{*}$, respectively)
\begin{multline}\label{eq3}
\|X_{\varepsilon}^{k}u\|^{2} \lesssim
\|\overline{\partial}_{M}^{*}X_{\varepsilon}^{k}u\|^{2} +
\|H_{m-1}X_{\varepsilon}^{k}u\|^{2}\\
= \|\overline{\partial}_{M}^{*}X_{\varepsilon}^{k}u\|^{2} +
\|H_{m-1}X_{\varepsilon}^{k}\overline{\partial}_{M}v\|^{2}
\lesssim \|[\overline{\partial}_{M}^{*}, X_{\varepsilon}^{k}]u\|^{2} +
\|[X_{\varepsilon}^{k},\overline{\partial}_{M}]v\|^{2} +
C_{\varepsilon}\|\overline{\partial}_{M}^{*}u\|_{k}^{2}\;.
\end{multline}
In the last estimate, we have used that the projection $H_{m-1}$ annihilates the range of
$\overline{\partial}_{M}$, and then that it is norm non-increasing. The commutator terms
on the right hand side of \eqref{eq3} can now be treated as before.

Finally, let $q = 0$. We want to use essentially the same argument, but with the role of $\overline{\partial}_{M}$ played by $\overline{\partial}_{M}^{*}$. When $m\geq 3$, there is no difficulty in doing so. When
$m=2$, note that the canonical solution $v$ to $\overline{\partial}_{M}^{*}v = u$ satisfies Sobolev estimates by the case $q=(m-1)=1$ already shown (this latter observation comes from \cite{BoasStraube91}). The proof of \eqref{main3*}, at the a priori level, is now also complete.

\smallskip
Now we prove \eqref{main1*}, also at the a priori level. First note that when $1\leq
q\leq (m-2)$, $H_{q}$ trivially satisfies Sobolev estimates, in view of \eqref{main2*}
(which we already established):
$\|Hu\|_{s} \lesssim \|Hu\| \leq \|u\| \leq \|u\|_{s}$.
We now look at the projections $S_{q}^{\prime}+H_{q}$ for $0\leq q\leq (m-2)$.
This argument is essentially from \cite[page 1580]{BoasStraube91}.

We need to estimate $\|(S_{q}^{\prime}+H_{q})u\|_{k}$ by $\|u\|_{k}$.
In view of \eqref{BS-estimate1}, it suffices to estimate
$\|X_{\varepsilon}^{k}(S_{q}^{\prime}+H_{q})u\|$ (note that
$\overline{\partial}_{M}(S_{q}^{\prime}+H_{q})u = 0$, and
$\overline{\partial}_{M}^{*}(S_{q}^{\prime}+H_{q})u = \overline{\partial}_{M}^{*}u$).
Now $(S_{q}^{\prime}+H_{q}-Id))u = S_{q}^{\prime\prime}u = \overline{\partial}_{M}^{*}v$,
where $v$ is the canonical solution to $\overline{\partial}_{M}^{*}$. Then $\|v\| \leq
\|(S_{q}^{\prime}+H_{q}-Id))u\|$, and $\|v\|_{k} \leq
\|(S_{q}^{\prime}+H_{q}-Id))u\|_{k}$, by \eqref{main3*} for $s=0$ and $s=k$, respectively.
This gives
\begin{multline}\label{eq4}
\left(X_{\varepsilon}^{k}(S_{q}^{\prime}+H_{q})u,
X_{\varepsilon}^{k}(S_{q}^{\prime}+H_{q})u\right) =
\left(X_{\varepsilon}^{k}S_{q}^{\prime\prime}u, X_{\varepsilon}^{k}(S_{q}+H_{q})u\right)
+
O_{\varepsilon}\left(\|u\|_{k}\|X_{\varepsilon}^{k}(S_{q}^{\prime}+H_{q})u\|\right) \\
= \left(\overline{\partial}_{M}^{*}X_{\varepsilon}^{k}v,
X_{\varepsilon}^{k}(S_{q}^{\prime}+H_{q})u\right)
- \left([\overline{\partial}_{M}^{*}, X_{\varepsilon}^{k}]v,
X_{\varepsilon}^{k}(S_{q}^{\prime}+H_{q})u\right) +
O_{\varepsilon}\left(\|u\|_{k}\|X_{\varepsilon}^{k}(S_{q}^{\prime}+H_{q})u\|\right) \\
= \left(X_{\varepsilon}^{k}v, [\overline{\partial}_{M},
X_{\varepsilon}^{k}](S_{q}^{\prime}+H_{q})u\right)
- \left([\overline{\partial}_{M}^{*}, X_{\varepsilon}^{k}]v,
X_{\varepsilon}^{k}(S_{q}^{\prime}+H_{q})u\right) +
O_{\varepsilon}\left(\|u\|_{k}\|X_{\varepsilon}^{k}(S_{q}^{\prime}+H_{q})u\|\right) ;
\end{multline}
we have used in the last equality that $\overline{\partial}_{M}$ annihilates the range of
$(S_{q}^{\prime}+H_{q})$.  The commutator terms in the last line of \eqref{eq4} can now
again be handled by the methods used above.

Because $S_{0}^{\prime}=0$, we have Sobolev estimates also for $H_{0}$ and thus for
$S_{0}^{\prime\prime}$. Because of the estimates for $H_{q}$ when $1\leq q\leq (m-2)$, we
also have estimates for $S_{q}^{\prime}$ for this range of $q$, and then also for
$S_{q}^{\prime\prime}$. Finally, when $q=(m-1)$, the argument is analogous, writing $(Id -
H_{m-1})u = S_{m-1}^{\prime}u$ as $\overline{\partial}_{M}v$. This completes the proof of
\eqref{main1*} (at the a priori level).

\smallskip
We next show that if $M$ happens to be strictly pseudoconvex, the estimates in Theorem \ref{fields-estimates} are genuine estimates. That is, we do not need to \emph{assume} that the forms on the left hand sides are in the appropriate Sobolev spaces. 
\cite{FollandKohn72}, Proposition 5.4.11 and Theorem 5.4.12
imply that (in the strictly pseudoconvex case) when $\overline{\partial}_{M}u$ and $\overline {\partial}_{M}^{*}u$ are in $W^{k}_{(0,q+1)}(M)$ and $W^{k}_{(0,q-1)}(M)$, respectively, then indeed $u \in W^{k}_{0,q)}(M)$. This takes care of \eqref{main2*}. When $1\leq q \leq (m-2)$, \eqref{main3*} follows from \eqref{main2*} (because when $u\perp\mathcal{H}_{q}$, $\|u\|$ is dominated by the right hand side of \eqref{main3*}). More generally, when $m\geq 3$, these estimates follow at all form levels from the subelliptic estimates for $\Box_{M,q}$ for $1\leq q \leq (m-2)$ (\cite{FollandKohn72}, Theorem 5.4.12). When
$m=2$ the canonical solution operators are $1/2$--subelliptic, by \cite{Kohn85}. Only the statements for $\overline{\partial}_{M}$ are given, but the estimates for $\overline{\partial}_{M}^{*}$ can be obtained by these methods as well (\cite{Kohn13}). Thus when the right hand side of \eqref{main3*} is finite, $u$ does belong to $W^{k}_{(0,q)}(M)$. As for \eqref{main1*}, note that in view of \eqref{main2*} already established, all three projections take smooth forms to smooth forms. Therefore, the a priori estimate established is a genuine estimate on smooth forms. But smooth forms are dense in $W^{k}_{(0,q)}(M)$, so that the estimate carries over to $W^{k}_{(0,q)}(M)$.

\medskip
We are now ready to remove the a priori assumptions for all $M$ as in Theorem \ref{fields-estimates}.
Let $\rho$ be the function from section
\ref{plush-proof}. As noted there (Lemma \ref{nonzero}), we can think of $\rho$ as a one--sided defining function for $M$ on
$\widehat{M}$. Then there exists a constant $A>0$ such that for $\delta >0$ small enough,
the CR submanifolds $M_{\delta} := \{z\in \widehat{M}\,|\,\rho(z) + \delta e^{A|z|^{2}} =
0\}$ are strictly pseudoconvex (and oriented, of hypersurface type). The construction is
the same as in \cite{Kohn73} (or see \cite{Straube10}, Lemma 2.5). Namely, one can cover
$M$ with finitely many local coordinate patches and get a constant $A$ for each patch; the
maximum will then work globally. The $M_{\delta}$ are the level sets of $\rho_{A}:=
e^{-A|z|^{2}}\rho$. As in \cite{BoasStraube91}, we fix a set of diffeomorphisms between
$M$ and $M_{\delta}$ (always for $\delta>0$ small enough) by the flow along the gradient
of $\rho_{A}$ (on $\widehat{M}$). Locally, we have bases $\{L_{j,\delta}$, $1\leq (m-1)$,
for $T^{1,0}(M_{\delta})$ that vary smoothly with $\delta$, and corresponding dual bases
$\{\omega_{j,\delta}\,|\,1\leq j\leq (m-1)\}$. We can now transfer forms on $M$ to forms
on $M_{\delta}$ coefficient wise in these charts (this also involves fixing a suitable
partition of unity). For a form $u$ on $M$, we denote by $u_{\uparrow\delta}$ the form obtained on $M_{\delta}$.
Similarly, $v^{\downarrow\delta}$ denotes the form on $M$ obtained from a form $v$ on $M_{\delta}$.
The vector fields in the family $\{X_{\varepsilon}\}$ live
on $M$. But for each $\varepsilon$, we can extend $X_{\varepsilon}$ into $\widehat{M}$,
near $M$, so that we still have the assumptions in Theorem \ref{fields-estimates}, but
with \eqref{small-commutator1} holding in a (one--sided) neighborhood $U_{\varepsilon}$ of
$M$ in $\widehat{M}$. Observe that if $Z$ denotes a first order differential operator (a vector field)
with coefficients that are smooth on $\widehat{M}$ (up to $M$), tangential to the $M_{\delta}$, and $u_{\delta}$ are a family of forms on $M_{\delta}$ such that the $(u_{\delta})^{\downarrow\delta}$ converge weakly in $W^{k}_{(0,q)}(M)$ to $u$,
then $(Zu_{\delta})^{\downarrow\delta} - Zu$ tends to zero on $M$ as distributions. Finally, by what was said above, the estimates
\eqref{main1*} -- \eqref{main4*} hold on $M_{\delta}$, for $\delta >0$ small enough. Moreover,
inspection of the proofs shows that these estimates are uniform in $\delta$. This latter point
is crucial.

The argument now proceeds almost verbatim as in \cite[page 1581]{BoasStraube91}, with small modifications.
Consider \eqref{main1*} first. Let $u \in W^{k}_{(0,q)}(M)$. Then $u_{\delta} =
S_{q\delta}^{\prime}u_{\delta} + S_{q,\delta}^{\prime\prime}u_{\delta} + H_{q,\delta}u_{\delta}$, where the subscripts on the operators denote operators on $M_{\delta}$. Because of the uniform bounds on these
operators, the transfers of all three terms to $M$ are bounded in $W^{k}_{(0,q)}(M)$ by a multiple of $\|u\|_{k}$. By passing to a suitable subsequence, we may assume that all three terms converge weakly as $\delta \rightarrow 0$. Call the respective weak limits $w_{1}$, $w_{2}$, and $w_{3}$. Then $u = w_{1} + w_{2} + w_{3}$, and $\|w_{1}\|_{k} + \|w_{2}\|_{k} + \|w_{3}\|_{k} \lesssim \|u\|_{k}$. By the observation in the previous paragraph about first order differential operators acting on forms $u_{\delta}$, $w_{3} \in \mathcal{H}_{q}(M)$ (because $H_{q,\delta}u_{\delta} \in \mathcal{H}_{q}(M_{\delta}$). Also, $S_{q,\delta}^{\prime}u_{\delta} = \overline{\partial}_{M_{\delta}}v_{\delta}$, with $\|v_{\delta}\|_{k} \lesssim \|u_{\delta}\|_{k} \lesssim \|u\|$ ($v_{\delta}$ is the canonical solution on $M_{\delta}$). By passing to a further subsequence, we may assume that the $(v_{\delta})^{\downarrow\delta}$ also converge weakly in $W^{k}_{(0,q-1)}(M)$. Call this limit $v$. Then
$(\overline{\partial}_{M_{\delta}}(v_{\delta}))^{\downarrow\delta} - \overline{\partial}_{M}v$ converges to zero (as distributions), so that $w_{1} = \overline{\partial}_{M}v$ (since $\overline{\partial}_{M_{\delta}}(v_{\delta}) = S_{q,\delta}^{\prime}u_{\delta}$). That is, $w_{1}$ is in the range of $\overline{\partial}_{M}$. Similarly, $w_{2}$ is in the range of $\overline{\partial}_{M}^{*}$. Because the Hodge decomposition of $u$ is unique, it follows that $S_{q}^{\prime}u = w_{1}$, $S_{q}^{\prime\prime}u = w_{2}$, and $H_{q}u = w_{3}$. This establishes \eqref{main1*}.

The previous paragraph yields more. First, when $u \in Im(\overline{\partial}_{M})$, then $u = S_{q}^{\prime}u = \overline{\partial}_{M}v$. $v$ is the weak limit, in $W^{k}_{(0,q-1)}(M)$, of the forms $(v_{\delta})^{\downarrow\delta}$. Consequently, $v$ belongs to the range of $\overline{\partial}_{M}^{*}$
(by an argument analogous to the one used to show that $S_{q}^{\prime}u \in Im(\overline{\partial}_{M})$ in the previous paragraph). In other words, $v$ is the canonical solution to $\overline{\partial}_{M}v = u$. This shows that the canonical solution operators to $\overline{\partial}_{M}$ are continuous in $W^{k}(M)$. Continuity of the canonical solution operators to $\overline{\partial}_{M}^{*}$ is established in the same way. Second, when $1 \leq q \leq (m-2)$, $\|H_{\delta}u_{\delta}\|_{k} \lesssim \|u_{\delta}\| \lesssim \|u\|$, so that $\|w_{3}\|_{k} = \|H_{q}u\|_{k} \lesssim \|u\|$.

With these estimates in hand, \eqref{main2*} and \eqref{main3*} are immediate: $\|u\|_{k} \leq \|S_{q}^{\prime}u\|_{k} + \|S_{q}^{\prime\prime}u\|_{k} + \|H_{q}u\|_{k}$. This quantity is bounded by the right hand side of \eqref{main2*} (or \eqref{main3*}), in view of the last paragraph and the fact that $S_{q}^{\prime}u$ and $S_{q}^{\prime\prime}u$ are the canonical solutions to the $\overline{\partial}_{M}$ and $\overline{\partial}_{M}^{*}$--equations with right hand sides $\overline{\partial}_{M}u$ and $\overline{\partial}_{M}^{*}u$, respectively.

The proof of Theorem \ref{fields-estimates} is now complete.

\end{proof}

\bigskip
\emph{Acknowledgment}: The authors are grateful for very useful correspondence from Joseph Kohn and Andreea Nicoara concerning estimates for $\overline{\partial}_{M}$ and $\overline{\partial}_{M}^{*}$ in \cite{Kohn84, KohnNicoara06}, and from Alex Tumanov concerning the construction of the `strip' manifold $\widehat{M}$ via analytic discs. They also thank C.~Denson Hill for a discussion on finite dimensionality vs. triviality of cohomology groups that led them to reference \cite{Brink02}. Finally, they thank the referee for very helpful comments on the exposition.


\bigskip
\providecommand{\bysame}{\leavevmode\hbox to3em{\hrulefill}\thinspace}

\end{document}